\documentclass{article}
\usepackage{graphicx} 
\usepackage{silence}
\usepackage{comment}
\usepackage{enumitem}
\usepackage[T1]{fontenc}
\usepackage[utf8]{inputenc} 

\usepackage{amsmath,amsthm,amssymb,mathtools}
\usepackage{bm}
\usepackage{bbm}
\usepackage{esint}

\usepackage{xcolor}
\usepackage{geometry}

\usepackage{hyphenat}

\usepackage{hyperref} 
\usepackage[nameinlink,capitalise]{cleveref}

\usepackage{microtype}

\allowdisplaybreaks

\allowdisplaybreaks
\mathtoolsset{showonlyrefs}

\newcommand{\ee}{e}
\newcommand{\D}{\nabla}

\newcommand{\norm}[1]{\left\lVert#1\right\rVert}

\newtheorem{Thm}{Theorem}[section]
\newtheorem{Lem}[Thm]{Lemma}
\newtheorem{Rem}[Thm]{Remark}

\newtheorem{Def}[Thm]{Definition}
\newtheorem{Prop}[Thm]{Proposition}

\def\Om{\Omega}

\def\R{\mathbb R}

\def\N{\mathbb N}
\def\e{\varepsilon}

\def\Xint#1{\mathchoice
   {\XXint\displaystyle\textstyle{#1}}
   {\XXint\textstyle\scriptstyle{#1}}
   {\XXint\scriptstyle\scriptscriptstyle{#1}}
   {\XXint\scriptscriptstyle\scriptscriptstyle{#1}}
   \!\int}
\def\XXint#1#2#3{{\setbox0=\hbox{$#1{#2#3}{\int}$}
     \vcenter{\hbox{$#2#3$}}\kern-.5\wd0}}

\def\dashint{\Xint-}
\newcommand{\prodscal}[2]{\langle{#1},{#2}\rangle}

\title{Morrey estimates for the gradient\\ in non-linear variational transmission problems}
\author{Luca Esposito\\
 {\it Dipartimento di Matematica, 
University of Salerno}\\ {\it Via Giovanni Paolo II,
84084 Fisciano (SA),  Italy} \\{\it  e-mail: luesposi@unisa.it }
 \bigskip
 \\
  Lorenzo Lamberti\\
 {\it Institute \'Elie Cartan, Université de Lorraine, CNRS}\\ {\it F-54000 Nancy, France} \\{\it  e-mail: lorenzo.lamberti@univ-lorraine.fr}}
\date{November 2025}

\begin{document}

\maketitle

\begin{abstract}
We study a class of variational transmission problems driven by nonlinear energies
with discontinuous coefficients across a prescribed interface.
The model setting consists of integral functionals of the form
\[
\mathcal{F}(u;E)=\int_{\Omega}\sigma_E(x)\,F(\nabla u)\,dx,
\]
where the coefficient $\sigma_E$ takes two constant values on complementary regions
separated by a $C^1$ hypersurface, and the integrand $F$ satisfies standard
$p$-growth and monotonicity conditions with $p>2$.

In this nonlinear variational framework, we establish local Morrey-space
regularity for the gradient of local minimizers, proving that
$\nabla u\in L^{2,\lambda}_{\mathrm{loc}}(\Omega)$ for every $0\leq\lambda<n$,
provided $2<p<\frac{2n}{n-2}$.
The proof is based on quantitative decay estimates for the energy near the interface,
first obtained in a flat configuration and then extended to the general case by
a suitable approximation argument.
\end{abstract}
\section{Introduction}

Transmission problems arise naturally in the study of diffusive processes in
heterogeneous media, where the governing laws change abruptly across fixed interfaces.
Typical applications include elasticity theory, composite materials, and conductivity
phenomena. This kind of problems appear also in the study of thermal insulation of bodies under sources and prescribed boundary conditions (see for example \cite{AC,DPNST})
From a mathematical perspective, such problems are characterised by the presence of
piecewise-defined operators or energy densities, together with suitable transmission
conditions prescribed along the separating interface.
The first systematic investigation of transmission problems in elasticity theory dates
back to the seminal work of M.~Picone in the 1950s (see \cite{Pico}).

In this paper, we investigate a class of \emph{variational transmission problems}
associated with nonlinear integral functionals exhibiting discontinuous coefficients
across a prescribed interface.
Our focus is on the regularity properties of local minimizers, with particular attention
to the behaviour of the gradient near the interface.

Let $\Omega \subset \mathbb{R}^n$ be a bounded domain and let $E \subset \Omega$ be
a measurable subset.
We assume that the interface $\partial E$ is a hypersurface of class $C^1$.
We fix two constants $0 < \alpha < \beta < +\infty$ and define
\[
\sigma_E := \beta \mathbbm{1}_E + \alpha \mathbbm{1}_{E^c},
\]
where $\mathbbm{1}_E$ denotes the characteristic function of the set $E$ and $E^c$ is the complementary of $E$.
We consider the integral functional
\begin{equation}\label{MainF}
\mathcal{F}(u;E) := \int_{\Omega} \sigma_E(x)\, F(\nabla u)\, dx,
\end{equation}
defined for functions $u \in W^{1,p}(\Omega)$.

The integrand $F \colon \mathbb{R}^n \to \mathbb{R}$ is assumed to belong to
$C^2(\mathbb{R}^n)$ and to satisfy standard $p$-growth and monotonicity conditions:
for all $\xi, \eta \in \mathbb{R}^n$,
\begin{equation}
\label{Monotonicity}\tag{H1}
\langle \nabla F(\xi) - \nabla F(\eta), \xi - \eta \rangle
\geq \nu \big( \mu^2 + |\xi|^2 + |\eta|^2 \big)^{\frac{p-2}{2}} |\xi - \eta|^2,
\end{equation}
\begin{equation}
\label{Growth}\tag{H2}
|\nabla F(\xi) - \nabla F(\eta)|
\leq L \big( \mu^2 + |\xi|^2 + |\eta|^2 \big)^{\frac{p-2}{2}} |\xi - \eta|,
\end{equation}
for some constants $\nu, L > 0$, $\mu \in (0,1]$, and $p > 2$.

Local minimizers of the functional $\mathcal{F}$ are understood in the usual
variational sense.

\begin{Def}
Let $u \in W^{1,p}(\Omega)$.
We say that $u$ is a \emph{local minimizer} of $\mathcal{F}(\cdot;\Omega)$ in $\Omega$
if, for every ball $B_r(x_0) \subset \Omega$ and every
$\phi \in W^{1,p}_0(B_r(x_0))$, one has
\[
\mathcal{F}(u;B_r(x_0)) \leq \mathcal{F}(u + \phi;B_r(x_0)).
\]
\end{Def}

Before stating our main result, let us place it in the context of the existing literature.
In the linear case $p = 2$, transmission problems with piecewise constant coefficients
have been extensively studied, and gradient regularity across sufficiently smooth
interfaces is by now classical (see, e.g.,~\cite{GT,LV,Grisvard}).
Such results follow from the elliptic regularity theory for divergence-form operators
with discontinuous coefficients (see for instance
\cite{Evans, GT, LM, Grisvard} and \cite[Theorem~7.53]{AFP}.
In this setting, local $W^{1,2}$- and H\"older-type regularity for the gradient of minimizers is well
understood under mild geometric assumptions on the interface.

More recently, fine regularity properties for linear transmission problems across
$C^{1,\alpha}$ interfaces have been obtained by Caffarelli, Soria-Carro, and Stinga
\cite{CSS}.

By contrast, the nonlinear case $p \neq 2$ remains far less developed.
To the best of our knowledge, the only available regularity results for nonlinear
transmission problems concern equations subject to \emph{additive} transmission
conditions on the interface.
In particular, in the recent work \cite{BPU}, BMO regularity of the gradient is obtained
for degenerate quasilinear equations under transmission conditions of the form
\[
g(|\nabla u_1|)\frac{\partial u_1}{\partial \nu}
-
g(|\nabla u_2|)\frac{\partial u_2}{\partial \nu}
= f \quad \text{on } \partial E.
\]

In contrast, the transmission condition naturally associated with the variational
functional \eqref{MainF} is of \emph{multiplicative} type, namely
\[
g_1(|\nabla u_1|)\frac{\partial u_1}{\partial \nu}
=
g_2(|\nabla u_2|)\frac{\partial u_2}{\partial \nu}
\quad \text{on } \partial E,
\]
which arises intrinsically from the Euler--Lagrange equations and expresses the
continuity of the nonlinear flux across the interface. We stress that this multiplicative transmission condition is intrinsic to the
variational structure of the functional \eqref{MainF}, as it arises directly from
the Euler--Lagrange equations and cannot be prescribed independently.
The presence of a discontinuous coefficient combined with nonlinear growth prevents
a direct application of standard techniques such as freezing arguments or difference
quotient estimates across the interface.
To the best of our knowledge, no Morrey-space regularity results for gradients of local
minimizers are available in this genuinely nonlinear variational transmission setting.

Denoting by $L^{{p},\lambda}$ the classical Morrey spaces, we can now state the main
result of this paper, which establishes local Morrey regularity for the gradient of
local minimizers under minimal geometric assumptions on the interface.

\begin{Thm}[Gradient regularity in Morrey spaces]\label{MainThm}
Let $u \in W^{1,p}(\Omega)$ be a local minimizer of  the functional $\mathcal{F}(\cdot;\Omega)$ defined
in\eqref{MainF}, and assume that
$\partial E$ is a hypersurface of class $C^1$.
Suppose that assumptions \emph{(H1)}--\emph{(H2)} are satisfied where $2<p<\frac{2n}{n-2}$.
Then, for every $0 \leq \lambda < n$,
\[
\nabla u \in L^{{p},\lambda}_{\mathrm{loc}}(\Omega).
\]
Moreover, for every $\Omega' \Subset \Omega$, there exists a constant
$C = C\big(n, p, \nu, L, \beta, E, \operatorname{diam}(\Omega),
\operatorname{dist}(\Omega', \partial\Omega)\big) > 0$
such that
\[
\|\nabla u\|_{L^{{p},\lambda}(\Omega')} \leq C.
\]
\end{Thm}
\begin{Rem}
Within the present variational transmission framework, Morrey regularity represents
the natural level of regularity that can be expected for the gradient of local
minimizers, in view of the discontinuous coefficients and the intrinsic transmission
condition.
\end{Rem}
The proof of Theorem~\ref{MainThm} is obtained as a consequence of the following
decay estimate for the energy.

\begin{Prop}\label{FJNL}
Let $u$ be a local minimizer of the functional $\mathcal{F}(\cdot;\Omega)$ defined
in \eqref{MainF} under the assumption of Theorem~\ref{MainThm}.
Then there exists a constant $0 < \tau_0 < 1$ such that the following holds:
for every $\tau \in (0,\tau_0)$ there exists $\varepsilon_0 = \varepsilon_0(\tau) > 0$
with the property that, if $B_r(x_0) \Subset \Omega$ and one of the following
conditions is satisfied,
\begin{enumerate}[label=(\roman*)]
\item $|E \cap B_r(x_0)| < \varepsilon_0 |B_r|$,
\item $|B_r(x_0) \setminus E| < \varepsilon_0 |B_r|$,
\item there exists a half-space $H$ such that
\[
\frac{|(E \Delta H) \cap B_r(x_0)|}{|B_r|} < \varepsilon_0,
\]
\end{enumerate}
then, for every $0 < \delta < n$, the estimate
\[
\int_{B_{\tau r}(x_0)} |\nabla u|^p \, dx
\leq
C_0 \tau^{n-\delta}
\left(
\int_{B_r(x_0)} |\nabla u|^p \, dx + r^n
\right)
\]
holds, where the constant $C_0$ depends only on
$n, \nu, L, \alpha,\beta, \delta$ and $\|\nabla u\|_{L^2(\Omega)}$.
\end{Prop}
\begin{Rem}
\label{Reg}
    For the proof of Theorem \ref{MainThm}, it would be sufficient to state the proposition only in case (i). We nevertheless include the other two cases for the sake of completeness.
\end{Rem}

We give an outline of the proof of Proposition~\ref{FJNL}. It relies on decay estimates for the gradient
near the interface, first established in a flat configuration, i.e.\ when the interface coincides
with a hyperplane.
Section~3 is devoted to this ``flat case''.

The starting point is a local boundedness estimate for the tangential gradient
$\nabla' u$, where $\nabla'=(\partial_{x_1},\dots,\partial_{x_{n-1}})$ denotes the vector of derivatives parallel to
the interface.
This is proved in Proposition~3.2 via a difference-quotient argument with test functions involving
only tangential increments.
A key ingredient is a Moser iteration applied to
$Z:=|\Delta_{i,h}u|^{m/p}$.
During the estimates, the full gradient $\nabla u$ enters through the structure conditions
\emph{(H1)}--\emph{(H2)} and is treated as a weight; this is the origin of the restriction
$2<p<\frac{2n}{n-2}$.

The bound obtained in Proposition~3.2 is preliminary, since the decay exponent in the estimates is not optimal and
the constant may depend on $\mu$.
To obtain the scale-invariant Lipschitz estimate for $\nabla' u$ stated in Proposition~3.4,
we perform a rescaling argument, which requires uniformity with respect to $\mu$.
This uniformity is proved in Lemma~3.3, exploiting the boundedness of $\nabla' u$ and a Poincar\'e-type
inequality for $Z:=|\partial_i u|^{m/p}$, $i=1,\dots,n-1$, which is an admissible Sobolev function by
Theorem~3.1.

Once the tangential Lipschitz bound is available, we prove the Morrey estimate in the flat case in
Theorem~3.5 by a comparison argument (Section~3.2).
Finally, the general case follows from a standard flattening/approximation procedure using the $C^1$
regularity of the interface, along the lines of the quadratic theory (see, e.g.,~\cite{E,FJ}).

\section{Auxiliary results and notation}
{
Throughout the paper we set
\[
    V^{2} := \mu^{2} + |\nabla u|^{2},
\]
where \(\nabla u\) denotes the weak gradient of \(u\).
}
\begin{Def}\label{MorreySpace}
Let $\Omega \subset \mathbb{R}^n$ be an open set, $1 \le p < \infty$, and $0 \le \lambda \le n$.
The \emph{Morrey space}  $L^{p,\lambda}(\Omega)$
is the set of all functions $f \in L^p_{\mathrm{loc}}(\Omega)$ such that
\[
\|f\|_{L^{p,\lambda}(\Omega)}
:= \sup_{x \in \Omega,\; r>0}
\left( r^{-\lambda} \int_{\Omega \cap B(x,r)} |f|^p \, dy \right)^{1/p}
< \infty,
\]
where $B(x,r)=\{y \in \mathbb{R}^n : |y-x|<r\}$.
\end{Def}
We will make use of the following iteration lemma that can be found in \cite[Lemma 2.VIII]{Cam}.
\begin{Lem}\label{Iter}
    Let $\phi\colon(0,r]\rightarrow\R$ be a nonnegative function such that for every $\sigma\in(0,r]$, $t\in(0,1)$ and $\epsilon>0$
    \begin{equation*}
        \phi(t\sigma)\leq\big(At^{\lambda}+B\epsilon\big)\phi(\sigma)+\epsilon^{-\beta}K\sigma^\mu,
    \end{equation*}
    where $0<\mu<\lambda$, $A>0$, $B\geq 0$, $K\geq 0$ and $\beta\geq 0$. Then for every $\delta<\lambda-\mu$ it holds that
    \begin{equation*}
        \phi(t\sigma)\leq (1+A)t^{\lambda-\delta}\phi(\sigma)+KM(t\sigma)^\mu,
\end{equation*}
where $M=M(A,B,\delta,\lambda,\mu,\beta)$ is a positive constant.
\end{Lem}

\begin{Lem}[Standard $p$-growth bounds]\label{lem:p-growth}
Assume \((H1)\)--\((H2)\). Then there exists a constant $C=C(n,p,\nu,L)>0$ such that for all $\xi\in\R^n$
\[
C^{-1}(\mu^2+|\xi|^2)^{\frac p2}-C\mu^p \le F(\xi)\le C(\mu^2+|\xi|^2)^{\frac p2}+C\mu^p .
\]
\end{Lem}
The following results on higher integrability and local Hölder continuity
of minimizers are classical. Since their proofs are standard, we omit them and refer the reader to
\cite{Giu,GiaMar}.
{
\begin{Thm}[Higher integrability]
\label{HigherInt}
Let $u \in W^{1,p}(\Omega)$ be a local minimizer of the functional \eqref{MainF}.
There exist $s > 1$ and $C = C(n,p,\nu,L,\beta)$ such that for any ball
$B_{2r}(x_{0})\subset\subset\Omega$
\begin{equation}
    \int_{B_r(x_0)}|\nabla u|^{sp}\,dx\leq C\bigg(\int_{B_{2r}(x_0)}\big(|\nabla u|^p+\mu^p\big)\,dx
    \bigg)^s.
\end{equation}
\end{Thm}
\begin{Thm}[H\"older continuity]Let $u \in W^{1,p}(\Omega)$ be a local minimizer of the functional \eqref{MainF}. Then
\begin{enumerate}[label=(\roman*)]
\item For any open set $\Omega'\Subset \Omega$ the quantity $\norm{u}_{L^{\infty}(\Omega')}$ is bounded by a constant $C = C(n,p,\nu,L,\beta)\norm{u}_{L^2(\Omega)}$.
\item u is locally H\"older continuous in $\Omega$.
\end{enumerate}
\end{Thm}
}

\section{Flat case}
{
In this section, we establish a local boundedness estimate for the gradient of a minimizer of the functional \eqref{MainF}, hereafter denoted by $\mathcal{F}$, in a ball $B_r(x_0)$ centered at a point $x_0 \in \partial E$, assuming that the interface $\partial E$ is flat in $B_r(x_0)$.

Up to a rotation, we may assume that the inner normal $\nu_E$ is aligned with the $e_n$-direction.
In particular,
\[
E = \{\, x \in \Omega : \prodscal{x-x_0}{e_n}> 0 \,\}.
\]
For $B_r(x_0)\subset\Omega$, we define
\[
B_r(x_0)^{\pm}= B_r(x_0)\cap \{\, x_n \gtrless (x_0)_n \,\},
\qquad
\Gamma_r = B_r(x_0)\cap \{\, x_n = (x_0)_n \,\}.
\]
Then it follows that
\[
B_r(x_0)\cap E = B_r(x_0)^+ .
\]
For the sake of readability we denote
$$
A=\nabla F\quad  \text{ and}\quad \sigma=\sigma_E.
$$
Moreover, we use the notation
\[
    \nabla' := (\partial_{x_1}, \ldots, \partial_{x_{n-1}})
\]
for the tangential gradient, i.e., the gradient with respect to the first \(n-1\) variables.
\subsection{Lipschitz estimate for \texorpdfstring{$\nabla' u$}{grad u}}
The technique involved in the proof makes use of difference quotients. Let us set the notations.
\subsection*{Difference quotients.}
For $s\in\{1,\dots,n\}$ and $h\in\R\setminus\{0\}$ we define the difference quotient
\[
\Delta_{s,h}f(x):=f(x+h\,\ee_s)-f(x),\qquad \delta_{s,h}f(x):=\frac{\Delta_{s,h}f(x)}{h}.
\]
The translate of $f$ in the direction of $e_s$ will be denoted by
\[
f_{h,\ee_s}(x):=f(x+h\,\ee_s).
\]
In order to guarantee that $x$ and $x+h\ee_s$ belong to $B_r(x_0)$ we work in
\[
B_{r-|h|}(x_0):=\{x\in\R^n:\ |x-x_0|<r-|h|\}\subset B_r(x_0).
\]
}

We remark that a  {local} minimizer $u\in W^{1,p}(B_r(x_0))$ of $\mathcal{F}$ satisfies the following weak Euler-Lagrange equation:
\begin{equation}
\label{EQEulLag}
\int_{B_r(x_0)}\sigma(x)\,\prodscal{A(\D u)}{\D\phi}\,dx=0,\quad\forall\phi\in W^{1,p}_0(B_r(x_0)).
\end{equation}

In the flat case, the existence of second derivatives of the minimizer $u$,  {separately in $B_r^+$ and $B_r^-$, up to the boundary $\Gamma_r$, can be established by standard arguments; we refer the reader to \cite[Section 8.4]{Giu}. In our setting, one can additionally prove a further property — useful in what follows — namely that the tangential gradient $\nabla'u$ admit a gradient also across the interface $\Gamma_r$.}

\begin{Thm}\label{DS}
Let $u$ be a minimizer of $\mathcal{F}$ in $B_r(x_0)$. Then $\partial_i u\in W^{1,2}_{loc}(B_{r})$ for every $i=1\dots n-1$. Moreover the following properties hold:
\begin{enumerate}[label=\roman*)]
\item $u \in W^{2,2}_{\mathrm{loc}}(B_r^+)$ e $u \in W^{2,2}_{\mathrm{loc}}(B_r^-)$;
\item for every $x_0 \in \Gamma_r$ and $R>0$ such that $B_{4R}(x_0)\subset B_r$, it holds that
\begin{align}
\int_{B_{R/2}} V^{p-2}\,|\nabla\nabla' u|^2 \,dx
&\leq \frac{C}{R^2}\int_{B_{2R}} V^p \,dx, \label{eq:tangential-estimate}
\end{align}
where $C=C\!\left(\tfrac{L}{\nu},\,n,\,p,\,\alpha,\,\beta,\mu\right)$ is a positive constant.
\end{enumerate}
\end{Thm}

\begin{proof}
Without loss of generality, we may assume that $x_0=0$.
Let $\rho\in\bigl(0,\tfrac{r}{4}\bigr)$ and choose a cutoff function
$\eta\in C_c^\infty(B_\rho)$ such that
\[
\eta \equiv 1 \ \text{on } B_{\rho/2}, \qquad
0\le \eta \le 1, \qquad
|\nabla\eta|\le \frac{C}{\rho},
\]
for a positive constant $C$.
Fix $i\in\{1,\dots,n-1\}$ and let $h\in\mathbb{R}\setminus\{0\}$ satisfy $|h|<\rho/2$.
We define the test function
\[
\phi:=\Delta_{i,-h}\bigl(\eta^2\,\Delta_{i,h}u\bigr)\in W^{1,p}_0(B_\rho).
\]

\medskip
By the Euler--Lagrange equation satisfied by $u$, using the change of variables
$y=x-he_i$ and observing that $\sigma$ depends only on $x_n$
(hence $\Delta_{i,h}\sigma=0$ for $i\in \{1,\dots,n-1\}$), we obtain
\begin{align}
\label{eq15}
0
&=\int_{B_\rho}\sigma\,\langle A(\nabla u),\Delta_{i,-h}\nabla(\eta^2\Delta_{i,h}u)\rangle\,dx \notag\\
&=\int_{B_\rho}\sigma\,\langle\Delta_{i,h}A(\nabla u),
\nabla(\eta^2\Delta_{i,h}u)\rangle\,dx \notag\\
&=\int_{B_\rho}\sigma\,\langle\Delta_{i,h}A(\nabla u),
\eta^2\nabla\Delta_{i,h}u
+2\eta\,\Delta_{i,h}u\,\nabla\eta\rangle\,dx\\
& =\sum_{j=1}^n\int_{B_\rho}
\sigma\,\Delta_{i,h}A^j(\nabla u)
\bigl(\eta^2\partial_j\Delta_{i,h}u
+2\eta\,\Delta_{i,h}u\,\partial_j\eta\bigr)\,dx.
\end{align}
We represent the difference quotient of $A$ using the fundamental theorem of calculus:
\begin{align}
\label{eq16}
\Delta_{i,h}A^j(\nabla u)
&=\frac1h\int_0^1\frac{d}{dt}
A^j\bigl(\nabla u+th\,\Delta_{i,h}\nabla u\bigr)\,dt \notag\\
&=\sum_{k=1}^n
\left(\int_0^1
\partial_kA^j\bigl(\nabla u+th\,\Delta_{i,h}\nabla u\bigr)\,dt\right)
\partial_k\Delta_{i,h}u .
\end{align}

We introduce the auxiliary quantity
\[
W_{i,h}:=\bigl(\mu^2+|\nabla u(x)|^2+|\nabla u(x+he_i)|^2\bigr)^{1/2}.
\]
By the growth assumptions on $\nabla A$ and by \cite[Lemma~8.3]{Giu}, we infer that
\begin{equation}
\label{eq17}
\left|
\int_0^1
\partial_kA^j\bigl(\nabla u+th\,\Delta_{i,h}\nabla u\bigr)\,dt
\right|
\le C(p,L)\,W_{i,h}^{p-2},
\end{equation}
and, by the monotonicity assumption,
\begin{align}
\label{eq18}
\sum_{j,k=1}^n
\left(\int_0^1
\partial_kA^j\bigl(\nabla u+th\,\Delta_{i,h}\nabla u\bigr)\,dt\right)
\xi_j\xi_k
\ge C(p,\nu)\,W_{i,h}^{p-2}|\xi|^2
\end{align}
for every $\xi\in\mathbb{R}^n$.

\medskip
Combining \eqref{eq15}, \eqref{eq16} and \eqref{eq18} and applying H\"older's and Young's inequalities,
we obtain
\begin{align*}
\int_{B_\rho}W_{i,h}^{p-2}\eta^2|\nabla\Delta_{i,h}u|^2\,dx
&\le
C\int_{B_\rho}
W_{i,h}^{p-2}
|\nabla\Delta_{i,h}u|\,|\Delta_{i,h}u|\,|\nabla\eta|\,\eta\,dx \\
&\le
C\left(
\varepsilon\int_{B_\rho}W_{i,h}^{p-2}\eta^2|\nabla\Delta_{i,h}u|^2\,dx
+\frac{1}{\varepsilon\rho^2}
\int_{B_\rho}W_{i,h}^{p-2}|\Delta_{i,h}u|^2\,dx
\right),
\end{align*}
where $C=C\!\left(n,p,\frac{\beta}{\alpha},\frac{L}{\nu}\right)>0$.
Choosing $\varepsilon>0$ sufficiently small and absorbing the first term into the left-hand side, we deduce
\begin{align}\label{eq:weighted-caccioppoli}
\int_{B_{\rho/2}} W_{i,h}^{\,p-2}\,\eta^{2}\,|\nabla \Delta_{i,h}u|^{2}\,dx
&\le \frac{C}{\rho^{2}}\int_{B_{\rho}} W_{i,h}^{\,p-2}\,|\Delta_{i,h}u|^{2}\,dx \notag\\
&\le \frac{C\,h^{2}}{\rho^{2}}\int_{B_{2\rho}} V^{p}\,dx,
\end{align}
where we have used the estimate
\[
\int_{B_{\rho}} W_{i,h}^{\,p-2}\,|\Delta_{i,h}u|^{2}\,dx
\le C\,h^{2}\Big(\int_{B_{\rho}} W_{i,h}^{\,p}\,dx
+\int_{B_{\rho}}|\delta_{i,h}u|^{p}\,dx\Big)
\le C\,h^{2}\int_{B_{2\rho}}V^{p}\,dx.
\]
Dividing \eqref{eq:weighted-caccioppoli} by $h^{2}$ we obtain
\begin{equation}\label{eq:weighted-caccioppoli-delta}
\int_{B_{\rho/2}} W_{i,h}^{\,p-2}\,\eta^{2}\,|\nabla \delta_{i,h}u|^{2}\,dx
\le \frac{C}{\rho^{2}}\int_{B_{2\rho}} V^{p}\,dx.
\end{equation}

\medskip
Since $\mu^{p-2}\le W_{i,h}^{p-2}$, estimate \eqref{eq:weighted-caccioppoli-delta} yields a uniform bound
for $\{\delta_{i,h}\nabla u\}_{|h|<\rho/2}$ in $L^{2}(B_{\rho/2})$.
By the standard characterization of Sobolev spaces via difference quotients
(see \cite[Lemma~8.2]{Giu} or \cite[Theorem~7.11]{GiaMar}), applied to $u$ and 
$\partial_j u$,
it follows that $\partial_i u\in W^{1,2}(B_{\rho/2})$ and that, up to a subsequence,
\[
\delta_{i,h}u \to \partial_i u \quad\text{strongly in } L^{2}(B_{\rho/2}),
\qquad
\nabla\delta_{i,h}u \to \nabla\partial_i u \quad\text{strongly in } L^{2}(B_{\rho/2})\; \text{ and a.e}.
\]

\medskip
We now pass to the limit in \eqref{eq:weighted-caccioppoli-delta}.
Since $W_{i,h}\to V$ a.e.\ in $B_{\rho/2}$, by Fatou's lemma we infer
\[
\int_{B_{\rho/2}} V^{p-2}\,\eta^{2}\,|\nabla\partial_i u|^{2}\,dx
\le \liminf_{h\to0}\int_{B_{\rho/2}} W_{i,h}^{\,p-2}\,\eta^{2}\,|\nabla \delta_{i,h}u|^{2}\,dx.
\]
Combining the previous inequality with \eqref{eq:weighted-caccioppoli-delta}, we get
\[
\int_{B_{\rho/2}} V^{p-2}\,\eta^{2}\,|\nabla\partial_i u|^{2}\,dx
\le \frac{C}{\rho^{2}}\int_{B_{2\rho}} V^{p}\,dx.
\]
Summing over $i\in\{1,\dots,n-1\}$ and recalling that $\eta\equiv 1$ on $B_{\rho/2}$ yields
\[
\int_{B_{\rho/2}} V^{p-2}\,|\nabla\nabla' u|^{2}\,dx
\le \frac{C}{\rho^{2}}\int_{B_{2\rho}} V^{p}\,dx,
\]
which proves \emph{(ii)} (up to the relabelling $\rho=2R$).

\medskip
Finally, since $\sigma$ is constant in each of the sets $B_r^\pm$, the same argument can be repeated
inside $B_r^\pm$ also for $i=n$, yielding $\partial_n u\in W^{1,2}_{\rm loc}(B_r^\pm)$.
Therefore $u\in W^{2,2}_{\rm loc}(B_r^+)$ and $u\in W^{2,2}_{\rm loc}(B_r^-)$, proving \emph{(i)}.
\end{proof}

In what follows we assume that $n>2$ and $2<p<2^*:=\frac{2n}{n-2}$. Since the computation leading to the desired estimate is rather lengthy, we divide the proof into several steps. We first establish the local boundedness of the tangential gradient
$\nabla' u$, without keeping track of the precise dependence of the resulting estimate on the constants and the exponents. This preliminary result allows us, in the subsequent propositions, to work with powers of the gradient
$\nabla' u$ in place of difference quotients, thereby simplifying the exposition and focusing on the improvement of the estimate and on the precise dependence on the constants.

{
\begin{Prop}
\label{LimDerTan}
Let $u$ be a minimizer of $\mathcal{F}$ in $B_r(x_0)$, $x_0\in \Gamma_r$. Then $\D_i u\in L^\infty(B_{\frac{\rho}{2}}(x_0))$, for every $\rho\in(0,r)$ and $i\in\{1,\dots,n-1\}$.
\end{Prop}
\begin{proof}
Without loss of generality, we can assume that $x_0=0$. We divide the proof in three steps.
\vspace{3mm}\\
\indent
\textbf{Step 1:} \textit{A Caccioppoli type inequality combined with Sobolev--Poincar\'e inequality.}
\\ We start writing an equation involving $u$ and $\Delta_{i,h}u$. We show that the following equation holds:
\begin{equation}
\label{eq1}
    \int_{B_{\rho-|h|}}\sigma(x)\,\prodscal{A(\D u_{h,\ee_i})-A(\D u)}{\D\phi}\,dx=0,\quad\forall \phi\in W^{1,p}_0(B_{\rho-|h|}).
\end{equation}
For any \(\phi\in W^{1,p}_0(B_{\rho-|h|})\), the translated function
\(\phi_{-h,\ee_i}\in W^{1,p}_0(B_\rho)\) can be used in the
Euler--Lagrange equation \eqref{EQEulLag}. Being $\Delta_{i,h}\sigma (x)=0$, a tangential translation then yields
\[
\int_{B_{\rho-|h|}}\sigma(x)\,\langle A(\nabla u_{h,\ee_i}),\nabla\phi\rangle\,dx=0.
\]
Subtracting the same identity with \(A(\nabla u)\), we obtain
\[
\int_{B_{\rho-|h|}}\sigma(x)\,\langle A(\nabla u_{h,\ee_i})-A(\nabla u),\nabla\phi\rangle\,dx=0,
\]
that is, \eqref{eq1}.\\
\indent Now we choose a suitable test function in \eqref{eq1}.
For any $m\geq p$ we denote $\tilde{m}:=\frac{2m}{p}\geq 2$ and define $\psi=\eta^2|\Delta_{i,h}u|^{\tilde{m}-2}\Delta_{i,h}u\in W^{1,p}_0(B_{\rho-|h|})$. It holds that
\begin{equation*}
\D\psi=2\eta|\Delta_{i,h}u|^{\tilde{m}-2}\Delta_{i,h}u\D\eta+(\tilde{m}-1)\eta^2|\Delta_{i,h}u|^{\tilde{m}-2}\D(\Delta_{i,h}u).
\end{equation*}
Plugging $\psi$ in \eqref{eq1}, we get
\begin{align*}
    &(\tilde{m}-1)\int_{B_{\rho-|h|}}\sigma(x)\,\eta^2|\Delta_{i,h}u|^{\tilde{m}-2}\prodscal{A(\D u_{h,\ee_i})-A(\D u)}{\D(\Delta_{i,h}u)}\,dx\\
    & =-2\int_{B_{\rho-|h|}}\sigma(x)\,\eta|\Delta_{i,h}u|^{\tilde{m}-2}\Delta_{i,h} u\prodscal{A(\D u_{h,\ee_i})-A(\D u)}{\D\eta}\,dx.
\end{align*}
Setting
\begin{equation*}
    W_{i,h}=(\mu^2+|\D u|^2+|\D u_{h,\ee_i}|^2)^{\frac{1}{2}},
\end{equation*}
and observing that $\tilde{m}-1\geq 1$, we apply \eqref{Monotonicity}, \eqref{Growth}, Holder's and Young's inequalities and infer that
\begin{align}
    & \alpha\nu\int_{B_{\rho-|h|}}W^{p-2}_{i,h}\eta^2|\Delta_{i,h}u|^{\tilde{m}-2}|\D(\Delta_{i,h}u)|^2\,dx\nonumber\\
    & \leq 2\beta L\int_{B_{\rho-|h|}}\eta|\Delta_{i,h}u|^{\tilde{m}-1}W_{i,h}^{p-2}|\D(\Delta_{i,h}u)||\D\eta|\,dx\nonumber\\
    & \leq 2\beta L\bigg[\frac{\varepsilon}{2}\int_{B_{\rho-|h|}}W_{i,h}^{p-2}\eta^2|\Delta_{i,h}u|^{\tilde{m}-2}|\D(\Delta_{i,h}u)|^2\,dx+\frac{1}{2\varepsilon}\int_{B_{\rho-|h|}}W_{i,h}^{p-2}|\Delta_{i,h}u|^{\tilde{m}}|\D\eta|^2\,dx \bigg],
\end{align}
for any $\varepsilon>0$.
Choosing
\[
\varepsilon:=\frac{\alpha\nu}{2\beta L},
\]
so that $\alpha\nu-\beta L\varepsilon=\alpha\nu/2$, we can absorb the first term on the right-hand side into the left-hand side and obtain
\begin{equation}
\label{eq12}
    \int_{B_{\rho-|h|}} W_{i,h}^{p-2}\eta^2
    |\Delta_{i,h}u|^{\tilde m-2}|\nabla(\Delta_{i,h}u)|^2\,dx
    \le
    C\!\left(\frac{\beta}{\alpha},\frac{L}{\nu}\right)
    \int_{B_{\rho-|h|}} W_{i,h}^{p-2}
    |\Delta_{i,h}u|^{\tilde m}|\nabla\eta|^2\,dx,
\end{equation}
where
\[
C = 4\left(\frac{\beta}{\alpha}\right)^2\left(\frac{L}{\nu}\right)^2.
\]
Defining
\begin{equation*}
Z:=|\Delta_{i,h}u|^{\frac{m}{p}},
\end{equation*}
so that
\begin{equation*}
    |\D Z|^2=\bigg(\frac{m}{p}\bigg)^2|\Delta_{i,h} u|^{\tilde{m}-2}|\D(\Delta_{i,h} u)|^2,
\end{equation*}
using \eqref{eq12} we get a Caccioppoli type inequality for $Z$,
\begin{align*}
    \int_{B_{\rho-|h|}}W_{i,h}^{p-2}|\D(\eta Z)|^2\,dx
    & \leq 2\int_{B_{\rho-|h|}}W_{i,h}^{p-2}|\D\eta|^2Z^2\,dx+2\int_{B_{\rho-|h|}}W_{i,h}^{p-2}|\D Z|^2\eta^2\,dx\\
    & \leq C\bigg(p,\frac{\beta}{\alpha},\frac{L}{\nu}\bigg)\big(1+m^2\big)\int_{B_{\rho-|h|}}W_{i,h}^{p-2}|\D\eta|^2Z^2\,dx.
\end{align*}
We then combine this inequality with the Sobolev--Poincar\'e inequality.
Let us fix $\frac{\rho}{2}\leq \rho''<\rho'\leq\rho$. We choose $\eta\in C^1_c(B_{\rho'})$ such that $\eta=1$ on $B_{\rho''}$ and $|\D\eta|\leq \frac{C}{\rho'-\rho''}$, for some positive constant $C$.
Applying the Sobolev--Poincar\'e inequality, remarking that $\mu^{p-2}\leq W^{p-2}_{i,h}$ and making use of H\"older's inequality, we get
\begin{align*}
    \bigg(\dashint_{B_{\rho''}}|Z|^{2^*}\,dx\bigg)^{\frac{1}{2^*}}
    & \leq \bigg(\dashint_{B_{\rho'}}|\eta Z|^{2^*}\,dx\bigg)^{\frac{1}{2^*}}\leq C(n)\rho'\bigg(\dashint_{B_{\rho'}}|\D(\eta Z)|^2\,dx\bigg)^{\frac{1}{2}}\\
    & \leq C(n)\mu^{\frac{2-p}{2}}\rho'\bigg(\dashint_{B_{\rho'}}W^{p-2}_{i,h}|\D(\eta Z)|^2\,dx\bigg)^{\frac{1}{2}}\\
    & \leq C\bigg(n,p,\frac{\alpha}{\beta},\frac{\nu}{L}\bigg)\big(1+m^2\big)^{\frac{1}{2}}\mu^{\frac{2-p}{2}}\rho' \bigg(\dashint_{B_{\rho'}}W_{i,h}^{p-2}|\D\eta|^2Z^2\,dx\bigg)^{\frac{1}{2}}\\
    & \leq C\bigg(n,p,\frac{\alpha}{\beta},\frac{\nu}{L}\bigg)\big(1+m^2\big)^{\frac{1}{2}}\mu^{\frac{2-p}{2}}\frac{\rho'}{\rho'-\rho''} \bigg(\dashint_{B_{\rho'}}W_{i,h}^{p-2}Z^2\,dx\bigg)^{\frac{1}{2}}\\
    & \leq C\bigg(n,p,\frac{\alpha}{\beta},\frac{\nu}{L}\bigg)\big(1+m^2\big)^{\frac{1}{2}}\mu^{\frac{2-p}{2}}\frac{\rho'}{\rho'-\rho''} \bigg(\dashint_{B_{\rho'}}W_{i,h}^p\,dx\bigg)^{\frac{p-2}{2p}}\bigg(\dashint_{B_{\rho'}}Z^p\,dx\bigg)^{\frac{1}{p}}.
\end{align*}
Since
\begin{align*}
\bigg(\dashint_{B_{\rho'}}W_{i,h}^p\,dx\bigg)^{\frac{p-2}{2p}}\leq C(n)\bigg(\frac{\rho}{\rho'}\bigg)^n\Phi_p(\rho)^{\frac{p-2}{2p}},
\end{align*}
where
\begin{equation*}
\Phi_p(\rho):=\dashint_{B_\rho}\big(1+|\D u|^2\big)^{\frac{p}{2}}\,dx,
\end{equation*}
we infer that
\begin{align}
    \label{eq13}
\bigg(\dashint_{B_{\rho''}}|Z|^{2^*}\,dx\bigg)^{\frac{1}{2^*}}\leq C\bigg(n,p,\frac{\beta}{\alpha},\frac{L}{\nu}\bigg)\big(1+m^2\big)^{\frac{1}{2}}\mu^{\frac{2-p}{2}}\frac{\rho'}{\rho'-\rho''}\bigg(\frac{\rho}{\rho'}\bigg)^n\Phi_p(\rho)^{\frac{p-2}{2p}}\bigg(\dashint_{B_{\rho'}}Z^p\,dx\bigg)^{\frac{1}{p}}.
\end{align}

\indent\textbf{Step 2:} \textit{Obtaining the starting inequality for Moser's iteration. } Let $\chi=\frac{2^*}{p}>1$ and
\begin{equation}
    \rho_k=\rho\bigg(\frac{1}{2}+\frac{1}{2^{k+1}}\bigg), \quad\forall k\in\N_0.
\end{equation}
We remark that $\{\rho_k\}_{k\in\N_0}$ is a decreasing sequence such that $\rho_k\rightarrow{\frac{\rho}{2}}$. We set $m_0=p$ and
\begin{equation}
    m_k=p\chi^k,\quad\forall k\in\N_0,
\end{equation}
so that $m_{k+1}=\chi m_k$, $\{m_k\}_{k\in\N_0}$ is an increasing sequence and $m_k\rightarrow{+\infty}$. Furthermore, we define
\begin{equation}
Y_k=\bigg(\dashint_{B_{\rho_k}}|\Delta_{i,h}u|^{m_k}\,dx\bigg)^{\frac{1}{m_k}},\quad\forall k\in\N_0.
\end{equation}
Applying \eqref{eq13} with $m=m_k$, $Z=|\Delta_{i,h}u|^{\frac{m_k}{p}}$ (so that $Z^{2^*}=|\Delta_{i,h}u|^{m_{k+1}}$), $\rho'=\rho_k$ and $\rho''=\rho_{k+1}$, we deduce that
\begin{equation*}
    Y_{k+1}^{\frac{m_{k+1}}{2^*}}\leq \Lambda_k\Phi_p(\rho)^{\frac{p-2}{2p}}Y_k^{\frac{m_k}{p}},\quad\text{where }\Lambda_k=C\bigg(n,p,\frac{\beta}{\alpha},\frac{L}{\nu}\bigg)\big(1+m_k^2\big)^{\frac{1}{2}}\mu^{\frac{2-p}{2}}\frac{\rho'}{\rho'-\rho''}\bigg(\frac{\rho}{\rho'}\bigg)^n,
\end{equation*}
for every $k\in\N_0$. Raising to the power $\frac{2^*}{m_{k+1}}=\frac{p}{m_k}=\frac{1}{\chi^k}$, we get
\begin{equation}
\label{eq14}
    Y_{k+1}\leq \Lambda_k^{\frac{1}{\chi^k}}\Phi_p(\rho)^{\frac{p-2}{2p}\frac{1}{\chi^k}}Y_k,\quad\forall k\in\N_0.
\end{equation}

\indent\textbf{Step 3:} \textit{Moser's iteration. }
Iterating \eqref{eq14} $N$ times starting from $k=0$, we get
\begin{equation}
    Y_{N}\leq \Bigg(\prod_{k=0}^{N-1}\Lambda_k^{\frac{1}{\chi^k}}\Bigg)\Phi_p(\rho)^{\frac{p-2}{2p}\sum_{k=0}^{N-1}\frac{1}{\chi^k}}Y_0.
\end{equation}

Letting $N\to+\infty$, we estimate the infinite product
\[
\prod_{k=0}^{\infty}\Lambda_k^{1/\chi^k}
= \exp\!\left(\sum_{k=0}^{\infty}\frac{1}{\chi^k}\log \Lambda_k\right).
\]
Recall that
\[
\Lambda_k
= C_0\,\mu^{\frac{2-p}{2}}\,(1+m_k^2)^{1/2}\,
\frac{\rho_k}{\rho_k-\rho_{k+1}}\left(\frac{\rho}{\rho_k}\right)^{n},
\qquad m_k=p\chi^k,
\]
where $C_0=C_0\!\left(n,p,\frac{\beta}{\alpha},\frac{L}{\nu}\right)$.
Since $\rho_k=\rho(\frac12+\frac{1}{2^{k+1}})$, we have
\[
\rho_k-\rho_{k+1}=\frac{\rho}{2^{k+2}}
\quad\Longrightarrow\quad
\frac{\rho_k}{\rho_k-\rho_{k+1}}\le c\,2^{k},
\qquad
\left(\frac{\rho}{\rho_k}\right)^n\le c,
\]
for a constant $c=c(n)>0$ independent of $k$. Hence, for a new constant
$C_1=C_1\!\left(n,p,\frac{\beta}{\alpha},\frac{L}{\nu}\right)$,
\[
\Lambda_k \le C_1\,\mu^{\frac{2-p}{2}}\,(1+m_k^2)^{1/2}\,2^{k}.
\]
Taking logarithms yields
\[
\log\Lambda_k \le \log C_1+\frac{2-p}{2}\log\mu+\frac12\log(1+m_k^2)+k\log 2.
\]
Therefore
\[
\sum_{k=0}^{\infty}\frac{1}{\chi^k}\log\Lambda_k
\le
\Big(\log C_1+\tfrac{2-p}{2}\log\mu\Big)\sum_{k=0}^{\infty}\frac{1}{\chi^k}
+\frac12\sum_{k=0}^{\infty}\frac{\log(1+m_k^2)}{\chi^k}
+(\log 2)\sum_{k=0}^{\infty}\frac{k}{\chi^k}.
\]
Since $\chi>1$, the geometric series $\sum_{k\ge0}\chi^{-k}$ converges and
$\sum_{k\ge0}k\,\chi^{-k}$ converges as well. Moreover, using $m_k=p\chi^k$ we have
\[
\log(1+m_k^2)\le \log(2p^2)+2k\log\chi,
\]
hence $\sum_{k\ge0}\chi^{-k}\log(1+m_k^2)<\infty$.
It follows that
\[
\sum_{k=0}^{\infty}\frac{1}{\chi^k}\log\Lambda_k <\infty,
\qquad\text{and thus}\qquad
\prod_{k=0}^{\infty}\Lambda_k^{1/\chi^k}\le C,
\]
where $C=C\!\left(n,p,\frac{\beta}{\alpha},\frac{L}{\nu},\mu\right)>0$.

and computing
\begin{equation*}
    \Phi_p(\rho)^{\frac{p-2}{2p}\sum_{k=0}^{+\infty}\frac{1}{\chi^k}}=\Phi_p(\rho)^{\frac{p-2}{2p}\frac{\chi}{\chi-1}},
\end{equation*}
we obtain
\begin{equation}
    \norm{\Delta_{i,h}u}_{L^\infty(B_{\frac{\rho}{2}})}\leq C\Phi_p(\rho)^{\frac{p-2}{2p}\frac{\chi}{\chi-1}}\norm{\Delta_{i,h}u}_{L^p(B_{\rho})}.
\end{equation}
Dividing the previous inequality by $h$ and letting $h\rightarrow 0$, we conclude that
\begin{align*}
    \norm{\D_i u}_{L^\infty(B_{\frac{\rho}{2}})}
    &\leq C\Phi_p(\rho)^{\frac{p-2}{2p}\frac{\chi}{\chi-1}}\norm{\D_i u}_{L^p(B_{\rho})}
    \leq C\Phi_p(\rho)^{\frac{p-2}{2p}\frac{\chi}{\chi-1}}\bigg(\dashint_{B_\rho}(\mu^2+|Du|^2)^{\frac p2}\,dx\bigg)^{\frac1p}\\
    &\leq C\Phi_p(\rho)^{\frac{p-2}{2p}\frac{\chi}{\chi-1}+\frac{1}{p}},
\end{align*}
which is the thesis.
\end{proof}
Having obtained local boundedness for $\nabla'u$, we can now proceed by using an appropriate test function to derive a uniform boundedness estimate for $\nabla'u$ independent of $\mu$.
}

\begin{Lem}
\label{Prop1Dec}
Let $u$ be a minimizer of $\mathcal{F}$ in $B_r(x_0)$, $x_0 \in \Gamma_r$. Then, setting $\chi=\frac{2^*}{p}$, for every $i\in\{1,\dots,n-1\}$ it holds that
\begin{equation*}
    \sup_{B_\frac{\rho}{2}(x_0)}|\D_i u|\leq C\bigg(\dashint_{B_\rho(x_0)}\big( {\mu^2}+|\D u|^2\big)^{\frac{p}{2}}\,dx\bigg)^{\frac{p-2}{2p}\frac{\chi}{\chi-1}+\frac{1}{p}},
\end{equation*}
for every $\rho\in\big(0,\frac{r}{2}\big)$ and for some positive constant $C=C\big(n,p,\frac{\beta}{\alpha},\frac{L}{\nu}\big)$ independent of $\mu$.
\end{Lem}

\begin{proof}
Let us fix $i\in\{1,\dots,n-1\}$ and $\rho<r$. Without loss of generality we may assume that $x_0=0$.
Let $m\ge p$ and set $\tilde m:=\frac{2m}{p}\ge2$.
Reasoning as in Proposition~\ref{LimDerTan}, we obtain the inequality \eqref{eq12},
which we recall here for the reader’s convenience:
\begin{equation}
    \int_{B_{\rho-|h|}} W_{i,h}^{p-2}\eta^2
    |\Delta_{i,h}u|^{\tilde m-2}|\nabla(\Delta_{i,h}u)|^2\,dx
    \le
    C\!\left(\frac{\beta}{\alpha},\frac{L}{\nu}\right)
    \int_{B_{\rho-|h|}} W_{i,h}^{p-2}
    |\Delta_{i,h}u|^{\tilde m}|\nabla\eta|^2\,dx .
\end{equation}

By the boundedness of $\partial_i u$ proved in Proposition~\ref{LimDerTan},
we can pass to the limit as $h\to0$ in the previous inequality. More precisely, by Proposition~\ref{LimDerTan} and Theorem~\ref{DS}\,(ii),
for every $i\in\{1,\dots,n-1\}$ we have, as $h\to0$,
\[
\delta_{i,h}u\to \partial_i u \quad \text{a.e.\ in } B_\rho,
\qquad
\delta_{i,h}u\to \partial_i u \quad \text{in } L^2_{\mathrm{loc}}(B_\rho),
\]
and moreover
\[
\nabla(\delta_{i,h}u)\to \nabla\partial_i u
\quad \text{in } L^2_{\mathrm{loc}}(B_\rho).
\]
Since $\partial_i u\in L^\infty_{\mathrm{loc}}(B_\rho)$ and $V^{p-2}\le 1+V^p\in L^1_{\mathrm{loc}}(B_\rho)$,
the right-hand side integrand is dominated by an $L^1_{\mathrm{loc}}$ function, hence dominated convergence applies.
Therefore, by weak lower semicontinuity on the left-hand side and dominated convergence on the right-hand side,
we may pass to the limit in \eqref{eq12} as $h\to0$, obtaining
\begin{equation}\label{eq3}
    \int_{B_\rho}V^{p-2}\eta^2|\D_i u|^{\tilde{m}-2}|\D\D_i u|^2\,dx
    \le
    C\int_{B_\rho}V^{p-2}|\D_i u|^{\tilde{m}}|\D\eta|^2\,dx .
\end{equation}

We now introduce the notation
\[
Z:=|\D_i u|^{\frac{m}{p}},
\]
which is well defined in $B_{\rho/2}$ by Proposition~\ref{LimDerTan}.
Moreover, by Theorem~\ref{DS} we have
\[
|\D Z|^2=\bigg(\frac{m}{p}\bigg)^2|\D_i u|^{\tilde m-2}|\D\D_i u|^2
\quad \text{a.e.\ in } B_{\rho/2}.
\]
With this notation, inequality \eqref{eq3} can be rewritten as
\begin{equation}\label{EstZ}
    \int_{B_\rho}V^{p-2}\eta^2|\D Z|^2\,dx
    \le
    C m^2\int_{B_\rho}V^{p-2}|\D\eta|^2Z^2\,dx ,
\end{equation}
where $C=C\big(p,\frac{\beta}{\alpha},\frac{L}{\nu}\big)>0$.

\medskip
\noindent Now we want to obtain the starting inequality for Moser's iteration.
Define
\[
Y:=1+Z,
\qquad
\Phi_p(s):=\dashint_{B_s}\big(\mu^2+|Du|^2\big)^{\frac p2}\,dx .
\]
Fix $\frac{\rho}{2}\le \rho''<\rho'\le \rho$ and choose
$\eta\in C_c^1(B_{\rho'})$ such that $\eta\equiv1$ on $B_{\rho''}$ and
$|\nabla\eta|\le c(\rho'-\rho'')^{-1}$.
Set $Z_1:=(Z-1)_+$. Then $Z_1\le Z$ and $\nabla Z_1=\nabla Z$ a.e.\ on $\{Z>1\}$.
Moreover,
\[
\{Z>1\}=\{|\D_i u|>1\}\subset\{|Du|>1\}.
\]
Since $p>2$, this implies $V\ge1$ and hence $V^{p-2}\ge1$ on $\operatorname{spt}Z_1$.
Using this observation and \eqref{EstZ}, we obtain
\begin{align*}
\int_{B_{\rho'}}|\nabla(\eta Z_1)|^2\,dx
&\le 2\int_{B_{\rho'}}|\nabla\eta|^2 Z_1^2\,dx
   +2\int_{B_{\rho'}}\eta^2|\nabla Z_1|^2\,dx \\
&\le 2\int_{B_{\rho'}}V^{p-2}|\nabla\eta|^2 Z^2\,dx
   +2\int_{B_{\rho'}}V^{p-2}\eta^2|\nabla Z|^2\,dx \\
&\le C(1+m^2)\int_{B_{\rho'}}V^{p-2}|\nabla\eta|^2 Z^2\,dx .
\end{align*}
By Sobolev's inequality and H\"older's inequality we deduce
\begin{align*}
\bigg(\dashint_{B_{\rho'}}|\eta Z_1|^{2^*}\,dx\bigg)^{\frac1{2^*}}
&\le C\,\rho'
\bigg(\dashint_{B_{\rho'}}|\nabla(\eta Z_1)|^2\,dx\bigg)^{\frac12} \\
&\le C\,\frac{\rho'}{\rho'-\rho''}\sqrt{1+m^2}\,
\Phi_p(\rho')^{\frac{p-2}{2p}}
\bigg(\dashint_{B_{\rho'}} Z^{p}\,dx\bigg)^{\frac1p}.
\end{align*}
Finally, since $Y\le 2+Z_1$ and $Y\ge1$, we have
$Y^{2^*}\le C(1+Z_1^{2^*})$ and $Y^p\le C(1+Z^p)$.
Using also the standard scaling inequality
\[
\bigg(\dashint_{B_{\rho''}} f^{2^*}\,dx\bigg)^{\frac1{2^*}}
\le C\Big(\frac{\rho'}{\rho''}\Big)^{\frac{n-2}{2}}
\bigg(\dashint_{B_{\rho'}} (\eta f)^{2^*}\,dx\bigg)^{\frac1{2^*}},
\qquad f\ge0,
\]
we conclude that
\begin{equation}\label{eq:RH}
\bigg(\dashint_{B_{\rho''}} Y^{2^*}\,dx\bigg)^{\frac1{2^*}}
\le
C \Big(\frac{\rho'}{\rho''}\Big)^{\frac{n-2}{2}}
\frac{\rho'}{\rho'-\rho''}
\sqrt{1+m^2}\,
\Phi_p(\rho')^{\frac{p-2}{2p}}
\bigg(\dashint_{B_{\rho'}} Y^{p}\,dx\bigg)^{\frac1p}.
\end{equation}

\medskip
\noindent With the previous inequality we can start the Moser iteration.
Let $\chi:=\frac{2^*}{p}>1$ and define
\[
\rho_k:=\rho\bigg(\frac12+\frac{1}{2^{k+1}}\bigg),\qquad
m_k:=p\chi^k,\qquad
Y_k:=1+|\D_i u|^{\frac{m_k}{p}},\qquad
N_k:=\bigg(\dashint_{B_{\rho_k}} Y_k^p\,dx\bigg)^{\frac{1}{m_k}}.
\]
Since
\[
Y_{k+1}^p
=\big(1+|\D_i u|^{\frac{m_{k+1}}p}\big)^p
\le
\big(1+|\D_i u|^{\frac{m_k}p}\big)^{p\chi}
=Y_k^{2^*},
\]
applying \eqref{eq:RH} with $m=m_k$, $\rho'=\rho_k$ and $\rho''=\rho_{k+1}$ and raising to the power $\frac{2^*}{m_{k+1}}=\frac{p}{m_k}=\frac{1}{\chi^k}$, we have
\[
N_{k+1}
\le
\Lambda_k^{\frac1{\chi^k}}
\Phi_p(\rho)^{\frac{p-2}{2p}\frac1{\chi^k}}
N_k ,
\]
where $\Lambda_k\le C\,2^k\sqrt{1+m_k^2}$.
Iterating and using the convergence of $\prod_k \Lambda_k^{1/\chi^k}$,
we obtain
\[
\|\D_i u\|_{L^\infty(B_{\rho/2})}
\le
C\,\Phi_p(\rho)^{\frac{p-2}{2p}\frac{\chi}{\chi-1}}
\bigg(\dashint_{B_\rho} Y_0^p\,dx\bigg)^{\frac1p},
\]
which concludes the proof of the lemma.
\end{proof}
Finally, exploiting a rescaling argument, we refine the Lipschitz estimate of the previous lemma, obtaining the correct power-law decay of the energy.
\begin{Prop} \label{FinaLip}
    Let $u$ be a minimizer of $\mathcal{F}$ in $B_r(x_0)$, $x_0\in \Gamma_r$. Then for every $i\in\{1,\dots,n-1\}$ it holds that
\begin{equation}
    \sup_{B_\frac{\rho}{2}(x_0)}|\D_i u|\leq C\bigg(\dashint_{B_\rho(x_0)}\big(\mu^2+|\D u|^2\big)^{\frac{p}{2}}\,dx\bigg)^{\frac{1}{p}},
\end{equation}
for every $0<\rho<\frac{r}{2}$ and for some positive constant $C=C\big(n,p,\frac{\beta}{\alpha},\frac{L}{\nu}\big)$ independent of $\mu$.
\end{Prop}

\begin{proof}
Let $0<\rho<r$ and set
\[
E:=\Big(\dashint_{B_\rho(x_0)}(\mu^2+|\nabla u|^2)^{\frac p2}\,dx\Big)^{\frac1p},
\qquad
v(y):=\frac{u(x_0+\rho y)}{\rho E},\quad y\in B_1 .
\]
Then $v\in W^{1,p}(B_1)$ and
\[
\nabla v(y)=E^{-1}\,\nabla u(x_0+\rho y)\qquad\text{for a.e. }y\in B_1.
\]

\medskip
\noindent\textbf{Step 1: Basic bounds and the rescaled coefficient.}
By definition of $E$ we have
\[
E^p=\dashint_{B_\rho(x_0)}(\mu^2+|\nabla u|^2)^{\frac p2}\,dx\ge \mu^p,
\]
hence $E\ge \mu$ and therefore
\[
\tilde\mu:=\frac{\mu}{E}\in(0,1].
\]
Moreover, since 

\[
(\mu^2+t^2)^{\frac p2}\ge \frac{1}{2}\big(\mu^p+t^p\big)\qquad\forall\,t\ge0,
\]

we infer
\[
E^p\ge \frac 12\,\dashint_{B_\rho(x_0)}\big(\mu^p+|\nabla u|^p\big)\,dx
\quad\Longrightarrow\quad
E^{-p}\dashint_{B_\rho(x_0)}|\nabla u|^p\,dx\le 2.
\]

Define the rescaled coefficient
\[
\sigma_\rho(y):=\sigma_E(x_0+\rho y),\qquad y\in B_1.
\]
In the flat case, $\sigma_\rho$ still takes only the two values $\alpha,\beta$
and the interface is $\{y_n=0\}$ (up to the fixed rotation), so we keep the same
notation $\sigma$ for $\sigma_\rho$ in what follows.

\medskip
\noindent\textbf{Step 2: The rescaled integrand $G$ and the Euler-Lagrange equation.}
We introduce the rescaled integrand
\[
G(\xi):=E^{-p}\,F(E\xi),\qquad \xi\in\mathbb R^n,
\]
and its gradient
\[
\tilde A(\xi):=\nabla G(\xi)=E^{1-p}\,A(E\xi),\qquad A:=\nabla F.
\]
A standard change of variables shows that, up to the constant factor $\rho^nE^p$,
the functional $\mathcal F$ on $B_\rho(x_0)$ is equivalent to the functional
\[
\widetilde{\mathcal F}(w;B_1):=\int_{B_1}\sigma(y)\,G(\nabla w)\,dy.
\]
In particular, the minimality of $u$ in $B_\rho(x_0)$ implies that $v$ is a
minimizer of $\widetilde{\mathcal F}(\cdot;B_1)$ in $B_1$, hence $v$ satisfies
the weak Euler--Lagrange equation
\[
\int_{B_1}\sigma(y)\,\langle \tilde A(\nabla v),\nabla\varphi\rangle\,dy=0,
\qquad \forall\,\varphi\in W^{1,p}_0(B_1).
\]

\medskip
\noindent\textbf{Step 3: Structure conditions for $\tilde A$.}
For every $\xi,\eta\in\mathbb R^n$ we have
\begin{align*}
\langle \tilde A(\xi)-\tilde A(\eta),\xi-\eta\rangle
&=E^{-p}\,\langle A(E\xi)-A(E\eta),E\xi-E\eta\rangle \\
&\ge \nu E^{-p}\big(\mu^2+|E\xi|^2+|E\eta|^2\big)^{\frac{p-2}{2}}\,E^2|\xi-\eta|^2 \\
&=\nu\big(\tilde\mu^2+|\xi|^2+|\eta|^2\big)^{\frac{p-2}{2}}|\xi-\eta|^2,
\end{align*}
and similarly
\begin{align*}
|\tilde A(\xi)-\tilde A(\eta)|
&=E^{1-p}|A(E\xi)-A(E\eta)| \\
&\le L E^{1-p}\big(\mu^2+|E\xi|^2+|E\eta|^2\big)^{\frac{p-2}{2}}\,E|\xi-\eta| \\
&=L\big(\tilde\mu^2+|\xi|^2+|\eta|^2\big)^{\frac{p-2}{2}}|\xi-\eta|.
\end{align*}
Therefore $\tilde A$ satisfies the same structural conditions with the same
constants $\nu,L$ and with parameter $\tilde\mu\in(0,1]$.

\medskip
\noindent\textbf{Step 4: Uniform energy bound for $v$.}
Using the identity $\nabla v=E^{-1}\nabla u(x_0+\rho y)$, we obtain
\begin{align*}
\dashint_{B_1}\big(\tilde\mu^2+|\nabla v|^2\big)^{\frac p2}\,dy
&=\dashint_{B_1}\Big(\frac{\mu^2}{E^2}+\frac{|\nabla u(x_0+\rho y)|^2}{E^2}\Big)^{\frac p2}\,dy \\
&=E^{-p}\dashint_{B_\rho(x_0)}\big(\mu^2+|\nabla u|^2\big)^{\frac p2}\,dx \\
&=E^{-p}\,E^p=1.
\end{align*}

\medskip
\noindent\textbf{Step 5: Apply Lemma~\ref{Prop1Dec} and scale back.}
Applying Lemma~\ref{Prop1Dec} (with radius $\rho=1$) to $v$ we find a constant
$C=C\big(n,p,\frac{\beta}{\alpha},\frac{L}{\nu}\big)$ independent of $\tilde\mu$
(and hence of $\mu$) such that, for every $i\in\{1,\dots,n-1\}$,
\[
\sup_{B_{1/2}}|\partial_i v|
\le
C\Big(\dashint_{B_1}\big(\tilde\mu^2+|\nabla v|^2\big)^{\frac p2}\,dy\Big)^{{\frac{p-2}{2p}\frac{\chi}{\chi-1}+\frac{1}{p}}}
= C.
\]
Finally, since $\partial_i u(x_0+\rho y)=E\,\partial_i v(y)$, changing variables
back gives
\[
\sup_{B_{\rho/2}(x_0)}|\partial_i u|
=E\,\sup_{B_{1/2}}|\partial_i v|
\le C\,E
=
C\Big(\dashint_{B_\rho(x_0)}(\mu^2+|\nabla u|^2)^{\frac p2}\,dx\Big)^{\frac1p},
\]
which is the desired estimate.
\end{proof}

\subsection{Morrey estimate in the flat case}
\begin{Thm}\label{FlatMorrey}
Let $u$ be a minimizer of $\mathcal{F}$ in $B_r(x_0)$. Then for every $\delta\in(0,1)$ it holds that
\begin{equation}
\int_{B_s^+}|\D u|^p\,dx\leq C\bigg(\frac{s}{\rho}\bigg)^{n-\delta}\int_{B_\rho^+}\big(1+|\D u|^2\big)^{\frac{p}{2}}\,dx,
\end{equation}
for every $0<s<\frac{\rho}{2}<\frac{r}{4}$ and for some positive constant $C=C\big(n,p,\frac{\beta}{\alpha},\frac{L}{\nu},\delta\big)$ independent of $\mu$.
\end{Thm}

\begin{proof}
Without loss of generality, we may assume that $x_0=0$. 
Since $u\in W^{1,p}(B_r)$ and is locally H\"older continuous up to $\Gamma_r$,
the trace $u(\cdot,0)$ is well defined pointwise and coincides with the Sobolev trace on $\Gamma_r$.
Moreover, by Proposition \ref{FinaLip}, $u(\cdot,0)\in W^{1,\infty}(B_\rho')$, for every $\rho<r/2$ and its constant extension in the $x_n$-direction belongs to
$W^{1,\infty}(B_{\rho})$, with distributional gradient $(\nabla' u(\cdot,0),0)$.
We define
\[
U(x):=u(x)-u(x',0)\,,\quad \forall x\in B_{\rho},
\]
so that $U=0$ on $\Gamma_{\rho}$. The function $U$ defined above vanishes identically on the flat boundary $\Gamma_\rho$,
which allows us to exploit classical boundary regularity theory up to the flat boundary.
At the same time, $U$ satisfies a perturbed Euler--Lagrange equation, where the perturbation
is encoded in the tangential gradient term $b$ defined below, and will be estimated accordingly.

Recalling that the tangential gradient $\nabla' u$ admits a trace on $\Gamma_{\rho}$
(e.g.\ by Theorem~3.1(ii) and Proposition~3.4), we set
\[
b(x'):=\bigl(\nabla' u(x',0),0\bigr)\in \mathbb{R}^n\,,\qquad x'\in \Gamma_{\rho},
\]
and extend $b$ to $B_\rho^+$ by $b(x',x_n):=b(x')$, so that
$\|b\|_{L^\infty(B_\rho^+)}=\|\nabla' u\|_{L^\infty(\Gamma_\rho)}$.
Let $0<\sigma<\rho<\frac{r}{2}$ with $B_\rho^+\subset E$ (hence the coefficient equals $\beta$ in $B_\rho^+$, and in particular in $B_\sigma^+$). $U$ solves
\begin{equation}\label{eq:Ueq-new}
\int_{B_\sigma^+}\langle A(\nabla U+b),\nabla\varphi\rangle\,dx=0,
\quad \forall\,\varphi\in W^{1,p}_0(B_\sigma^+).
\end{equation}
Now let $U_0\in W^{1,p}(B_\sigma^+)$ be the unique weak solution to the
unperturbed Euler--Lagrange equation associated with $F$ in $B_\sigma^+$,
with boundary datum $U$, namely
\begin{equation}\label{eq:U0eq-new}
\int_{B_\sigma^+}\langle A(\nabla U_0),\nabla\varphi\rangle\,dx=0
\quad \forall\,\varphi\in W^{1,p}_0(B_\sigma^+).
\end{equation}
Equivalently, $U_0$ is the unique minimizer of $\mathcal F_0$ in
$U+W^{1,p}_0(B_\sigma^+)$.

Set $W:=U-U_0\in W^{1,p}_0(B_\sigma^+)$. Subtracting \eqref{eq:Ueq-new}
from the above equation and testing with $\varphi=W$, we obtain
\begin{equation}\label{eq:diff-new}
\int_{B_\sigma^+}\langle A(\nabla U_0)-A(\nabla U+b),\nabla W\rangle\,dx=0.
\end{equation}

\medskip\noindent
\textbf{Step 1: a comparison estimate.}
Set $\xi:=\nabla U_0$ and $\eta:=\nabla U+b$. Since
\[
\nabla W=\nabla U-\nabla U_0 = (\eta-b)-\xi =-(\xi-\eta+b),
\]
from \eqref{eq:diff-new} we infer
\[
\int_{B_\sigma^+}\langle A(\xi)-A(\eta),\xi-\eta\rangle\,dx
=\int_{B_\sigma^+}\langle A(\xi)-A(\eta),b\rangle\,dx.
\]
By monotonicity \((H1)\) and growth \((H2)\), for a.e.\ $x\in B_\sigma^+$,
\[
\langle A(\xi)-A(\eta),\xi-\eta\rangle
\ge \nu\big(\mu^2+|\xi|^2+|\eta|^2\big)^{\frac{p-2}{2}}|\xi-\eta|^2,
\]
and
\[
|\langle A(\xi)-A(\eta),b\rangle|
\le L\big(\mu^2+|\xi|^2+|\eta|^2\big)^{\frac{p-2}{2}}|\xi-\eta|\,|b|.
\]
Applying Young's inequality to the right-hand side yields, for every $\varepsilon\in(0,1)$,
\[
|\langle A(\xi)-A(\eta),b\rangle|
\le \varepsilon\,\nu\big(\mu^2+|\xi|^2+|\eta|^2\big)^{\frac{p-2}{2}}|\xi-\eta|^2
+ C(\varepsilon)\big(\mu^2+|\xi|^2+|\eta|^2\big)^{\frac{p-2}{2}}|b|^2.
\]
Integrating and absorbing the $\varepsilon$--term, we obtain
\begin{equation}\label{eq:L2comp-new}
\int_{B_\sigma^+}\big(\mu^2+|\nabla U_0|^2+|\nabla U+b|^2\big)^{\frac{p-2}{2}}
|\nabla U_0-(\nabla U+b)|^2\,dx
\le C\int_{B_\sigma^+}\big(\mu^2+|\nabla U_0|^2+|\nabla U+b|^2\big)^{\frac{p-2}{2}}|b|^2\,dx.
\end{equation}
We now estimate $\int_{B_\sigma^+}|\nabla W|^p$.
Note that
\[
-\nabla W=(\nabla U_0-(\nabla U+b))+b.
\]
Hence, we have
\begin{align}
\int_{B_\sigma^+}|\nabla W|^p\,dx
&\le C\int_{B_\sigma^+}\big|\nabla U_0-(\nabla U+b)\big|^p\,dx + C\sigma^n\|b\|_{L^\infty(B_\sigma^+)}^p.
\label{eq:splitW}
\end{align}
Accordingly, by \eqref{eq:L2comp-new} and \eqref{eq:splitW} we deduce that
\begin{align*}
    \int_{B_\sigma^+}|\D W|^p
    & \leq C\bigg(\int_{B_\sigma^+}\big(\mu^2+|\nabla U_0|^2+|\nabla U+b|^2\big)^{\frac{p-2}{2}}|b|^2\,dx+ \sigma^n\|b\|_{L^\infty(B_\sigma^+)}^p\bigg)\\
    & \leq C\bigg(\varepsilon\int_{B_\sigma^+}\big(\mu^2+|\nabla U_0|^2+|\nabla U|^2\big)^{\frac{p}{2}}\,dx+C(\varepsilon) \sigma^n\|b\|_{L^\infty(B_\sigma^+)}^p\bigg),
\end{align*}
where $C(\varepsilon)=C\varepsilon^{-1}$.
\\
\noindent \textbf{Step 2: control of $\nabla U_0$ by $\nabla U$.}
By minimality of $U_0$ (using $U$ as a competitor) we have
\[
\int_{B_\sigma^+}F(\nabla U_0)\,dx \le \int_{B_\sigma^+}F(\nabla U)\,dx.
\]
By virtue of Lemma \ref{lem:p-growth} we infer
\begin{equation}\label{eq:U0byU-new}
\int_{B_\sigma^+}(\mu^2+|\nabla U_0|^2)^{\frac p2}\,dx
\le C\int_{B_\sigma^+}(\mu^2+|\nabla U|^2)^{\frac p2}\,dx + C\sigma^n\mu^p.
\end{equation}

\medskip\noindent
\textbf{Step 3: boundary gradient bound and energy decay for $U_0$.}
Since $U_0$ is a weak solution to the homogeneous quasilinear elliptic equation
in divergence form \eqref{eq:U0eq-new} in $B_\sigma^+$ and satisfies the homogeneous Dirichlet
condition $U_0=0$ on the flat portion of the boundary $\Gamma_\sigma$,
boundary $C^{1,\alpha}$ regularity theory applies.
In particular, by Lieberman~\cite[Theorem~1, p.~1203]{Lieberman1988},
there exist constants $\alpha_0\in(0,1/2)$ and $C=C(n,p,\nu,L)$ such that
\begin{equation}\label{eq:Lieberman-grad}
\|\nabla U_0\|_{L^\infty(B_{\alpha_0\sigma}^+)}
\le
C\left(\dashint_{B_\sigma^+}
(\mu^2+|\nabla U_0|^2)^{\frac p2}\,dx\right)^{\!\frac1p}.
\end{equation}
We point out that the constants $C,\alpha_0$ in \eqref{eq:Lieberman-grad}
depend only on $n,p,\nu,L$ and are uniform in $\mu\in(0,1]$.
Relying on this result, we now deduce a decay estimate for $\nabla U_0$,
which is a direct consequence of \eqref{eq:Lieberman-grad}.

\medskip

Fix $\tau\in(0,1)$. If $\tau\le \alpha_0$, then by \eqref{eq:Lieberman-grad} and the inclusion
$B_{\tau\sigma}^+\subset B_{\alpha_0\sigma}^+$ we have
\begin{align}
\int_{B_{\tau\sigma}^+}|\nabla U_0|^p\,dx
&\le |B_{\tau\sigma}^+|\,\|\nabla U_0\|_{L^\infty(B_{\alpha_0\sigma}^+)}^p \nonumber\\
&\le C\,\tau^n \int_{B_\sigma^+}
(\mu^2+|\nabla U_0|^2)^{\frac p2}\,dx,
\label{eq:decayU0-new-small}
\end{align}
where we used that $|B_{\tau\sigma}^+|=\tau^n |B_\sigma^+|$ and possibly enlarged the constant $C$.

If instead $\tau>\alpha_0$, then trivially
\[
\int_{B_{\tau\sigma}^+}|\nabla U_0|^p\,dx
\le \int_{B_\sigma^+}|\nabla U_0|^p\,dx,
\]
and since $\tau^n\ge \alpha_0^n$ we obtain
\begin{equation}\label{eq:decayU0-new-large}
\int_{B_{\tau\sigma}^+}|\nabla U_0|^p\,dx
\le \alpha_0^{-n}\,\tau^n
\int_{B_\sigma^+}|\nabla U_0|^p\,dx
\le C\,\tau^n
\int_{B_\sigma^+}(\mu^2+|\nabla U_0|^2)^{\frac p2}\,dx,
\end{equation}
where the last inequality follows from
$|\nabla U_0|^p\le (\mu^2+|\nabla U_0|^2)^{p/2}$.

Combining \eqref{eq:decayU0-new-small} and \eqref{eq:decayU0-new-large},
we conclude that for every $\tau\in(0,1)$ there exists a constant
$C=C(n,p,\nu,L)$ such that
\begin{equation}\label{eq:decayU0-new}
\int_{B_{\tau\sigma}^+}|\nabla U_0|^p\,dx
\le C\,\tau^n
\int_{B_\sigma^+}(\mu^2+|\nabla U_0|^2)^{\frac p2}\,dx .
\end{equation}

\medskip\noindent
\textbf{Step 4: decay for $\nabla U$ and Morrey iteration.}
Using $\nabla U=\nabla U_0+\nabla W$, for any $\tau\in(0,1)$ we have
\[
\int_{B_{\tau\sigma}^+}|\nabla U|^p\,dx
\le C\int_{B_{\tau\sigma}^+}|\nabla U_0|^p\,dx + C\int_{B_\sigma^+}|\nabla W|^p\,dx.
\]
By \eqref{eq:decayU0-new} and \eqref{eq:U0byU-new} we obtain, for every $\varepsilon\in(0,1)$,
\begin{align}
\int_{B_{\tau\sigma}^+}|\nabla U|^p\,dx
&\le C\,\tau^n\int_{B_\sigma^+}(\mu^2+|\nabla U_0|^2)^{\frac p2}\,dx
+ C\varepsilon\int_{B_\sigma^+}(\mu^2+|\nabla U_0|^2+|\nabla U|^2)^{\frac p2}\,dx
+ C(\varepsilon)\sigma^n\|b\|_{L^\infty(B_\sigma^+)}^p \nonumber\\
&\le C(\tau^n+\varepsilon)\int_{B_\sigma^+}(\mu^2+|\nabla U|^2)^{\frac p2}\,dx
+ C(\tau^n+\varepsilon)\sigma^n\mu^p
+ C(\varepsilon)\sigma^n\|b\|_{L^\infty(B_\sigma^+)}^p \nonumber\\
&\le C(\tau^n+\varepsilon)\int_{B_\sigma^+}|\nabla U|^p\,dx
+ C(\tau^n+\varepsilon)\sigma^n
+ C(\varepsilon)\sigma^n\|b\|_{L^\infty(B_\sigma^+)}^p .
\label{eq:preiter-new}
\end{align}
Fix $\delta\in(0,1)$. Since $\sigma\le \rho$, we may write $\sigma^n=\sigma^{n-\delta}\sigma^\delta\le \sigma^{n-\delta}\rho^\delta$.
Therefore \eqref{eq:preiter-new} implies
\begin{equation}\label{eq:iter-new}
\int_{B_{\tau\sigma}^+}|\nabla U|^p\,dx
\le C\bigl(\tau^n+\varepsilon\bigr)\int_{B_\sigma^+}|\nabla U|^p\,dx
+ C\varepsilon^{-1}\sigma^{n-\delta}\rho^\delta\Bigl(1+\|b\|_{L^\infty(B_\sigma^+)}^p\Bigr),
\end{equation}
for a suitable positive constant $C$ depending only on $n,p,\nu,L$.
Choosing $\varepsilon>0$ sufficiently small (depending only on $C,\delta$) and applying Lemma~\ref{Iter}
to $\phi(\sigma):=\int_{B_\sigma^+}|\nabla U|^p\,dx$ (with $\lambda=n$ and $\gamma=n-\delta$)
we infer that
\begin{equation}\label{eq:postiter-new}
\int_{B_{\tau\sigma}^+}|\nabla U|^p\,dx
\le C\,\tau^{n-\delta/2}\int_{B_\sigma^+}|\nabla U|^p\,dx
+ C(\tau\sigma)^{n-\delta}\rho^\delta\Bigl(1+\|b\|_{L^\infty(B_\sigma^+)}^p\Bigr),
\end{equation}
for all $\tau\in(0,1)$ and $0<\sigma<\rho$.

\medskip\noindent
\textbf{Step 5: conclusion for $\nabla u$.}
Choose $\sigma=\rho/2$ and $\tau=2s/\rho$ (so that $B_s^+=B_{\tau\sigma}^+$ and $\tau\in(0,1)$ whenever $0<s<\rho/2$).
Using \eqref{eq:postiter-new} and the inequality $|\nabla u|^p\le C\big(|\nabla U|^p+|b|^p\big)$, we obtain
\begin{align*}
\int_{B_s^+}|\nabla u|^p\,dx
&\le C\left(\frac{s}{\rho}\right)^{n-\delta}\int_{B_{\rho/2}^+}|\nabla U|^p\,dx
+ C\,s^{n-\delta}\rho^\delta\Bigl(1+\|b\|_{L^\infty(B_{\rho/2}^+)}^p\Bigr).
\end{align*}
Moreover, since $\nabla U=\nabla u-b$, we have
\[
\int_{B_{\rho/2}^+}|\nabla U|^p\,dx\le C\int_{B_{\rho/2}^+}|\nabla u|^p\,dx + C\rho^n\|b\|_{L^\infty(B_{\rho/2}^+)}^p
\le C\int_{B_{\rho}^+}\big(1+|\nabla u|^2\big)^{\frac p2}\,dx + C\rho^n\|b\|_{L^\infty(B_{\rho/2}^+)}^p.
\]
By the tangential Lipschitz estimate (Proposition~3.4),
\[
\|b\|_{L^\infty(B_{\rho/2}^+)}=\|\nabla'u\|_{L^\infty(\Gamma_{\rho/2})}
\le C\left(\dashint_{B_\rho}\big(\mu^2+|\nabla u|^2\big)^{\frac p2}\,dx\right)^{\!\frac1p}.
\]
Inserting these bounds and absorbing the lower-order contributions into the right-hand side,
we conclude that for every $0<s<\rho/2<\frac{r}{4}$,
\[
\int_{B_s^+}|\nabla u|^p\,dx
\le C\left(\frac{s}{\rho}\right)^{n-\delta}\int_{B_\rho^+}\big(1+|\nabla u|^2\big)^{\frac p2}\,dx,
\]
with $C=C\bigl(n,p,\tfrac{\beta}{\alpha},\tfrac{L}{\nu},\delta\bigr)$ independent of $\mu$.
This concludes the proof.
\end{proof}
\begin{Rem}\label{rem:scales}
Let the assumptions of Theorem~\ref{FlatMorrey} be satisfied.  
Then the estimate
\[
\int_{B_s^+}|\nabla u|^p\,dx
\leq C\bigg(\frac{s}{\rho}\bigg)^{n-\delta}
\int_{B_\rho^+}\big(1+|\nabla u|^2\big)^{\frac{p}{2}}\,dx
\]
holds for every $0<s<\rho<r$, possibly with a different constant $C>0$.

More precisely, the restriction $0<s<\rho/2<r/4$ in Theorem~\ref{FlatMorrey}
can be removed at the price of modifying the constant $C$ by a universal
factor depending only on $n$ and $\delta$. In particular, the resulting
constant still depends only on
$n,p,\frac{\beta}{\alpha},\frac{L}{\nu},\delta$ and is independent of $\mu$.
\end{Rem}

\section{Proof of Proposition 1.3}
\begin{proof}
Fix $\tau\in(0,\tau_0)$ (to be chosen small only depending on $n,p,\nu,L,\beta/\alpha$) and assume without loss of generality that $x_0=0$.
We prove that there exists $\e_0=\e_0(\tau)>0$ such that if $B_r(x_0)\Subset\Om$ and one of {\rm(i)}--{\rm(iii)}
holds, then for every $0<\delta<n$, then
\begin{equation}\label{eq:prop13-goal}
\int_{B_{\tau r}} |\nabla u|^p\, dx
\leq
C_0\tau^{n-\delta}\bigg( \int_{B_r} |\nabla u|^p \, dx
+r^n\bigg),
\end{equation}
where $C_0=C_0(n,p,\nu,L,\beta,\delta,\|\nabla u\|_{L^2(\Om)})>0$.\\
\indent 
\noindent\textbf{Step 1: the flat comparison in case (iii).}
Assume that (iii) holds true. Up to a rigid motion we may assume that $H=\{x_n>0\}$ and set
\[
E_H:=H\cap B_{r},\quad \sigma_H:=\beta\mathbbm 1_{E_H}+\alpha\mathbbm 1_{B_r\setminus E_H}.
\]
Let $u_H\in u+W^{1,p}_0(B_\frac{r}{2})$ be the unique minimizer of the functional
\[
\mathcal F_H(w;B_\frac{r}{2}):=\int_{B_\frac{r}{2}}\sigma_H(x)\,F(\nabla w)\,dx .
\]
Subtracting the Euler-Lagrange equation satisfied by $u$ from the one satisfied by $u_H$ and testing them with $u-u_H$, we get
\begin{align}
    \int_{B_{\frac{r}{2}}} \sigma_E
\big(DF(\nabla u)-DF(\nabla u_H)\big)
\cdot(\nabla u-\nabla u_H)\,dx
& =
\int_{B_{\frac{r}{2}}} (\sigma_H-\sigma_E)\,
DF(\nabla u_H)\cdot(\nabla u-\nabla u_H)\,dx
\end{align}
Using assumptions \eqref{Monotonicity} and \eqref{Growth} and H\"older's inequality, for every $\varepsilon>0$ it holds that
\begin{align}
    \int_{B_{\frac{r}{2}}}|\D u_H-\D u|^p\,dx
    & \leq C\int_{(E\Delta E_H)\cap B_{\frac{r}{2}}}\big(\mu^2+|\D u_H|^2\big)^{\frac{p-2}{2}}|\D u-\D u_H|\,dx\\
& \leq C\bigg(\frac{1}{\varepsilon}\int_{(E\Delta E_H)\cap B_{\frac{r}{2}}}\big(\mu^2+|\D u_H|^2\big)^{\frac{p}{2}}\,dx+\varepsilon\int_{B_\frac{r}{2}}|\D u-\D u_H|^p\,dx \bigg).
\end{align}
Choosing $\varepsilon$ sufficiently small, we get 
\begin{align}
    \int_{B_{\frac{r}{2}}}|\D u_H-\D u|^p\,dx
& \leq C\int_{(E\Delta E_H)\cap B_{\frac{r}{2}}}\big(\mu^2+|\D u_H|^2\big)^{\frac{p}{2}}\,dx.
\end{align}
Thus, the minimality of $u_H$ with respect to $u$, H\"older's inequality and Theorem \ref{HigherInt} yield
\begin{align}
\label{eqq1}
\int_{B_{\frac{r}{2}}}|\D u_H-\D u|^p\,dx& \leq \int_{(E\Delta E_H)\cap B_\frac{r}{2}} (1+|\nabla u|^2)^{\frac p2}\,dx\\
& \le \bigg(\frac{|(E\Delta E_H)\cap B_r|}{|B_r|}\bigg)^{1-\frac1s}|B_r|\bigg(\dashint_{B_r}(1+|\nabla u|^2)^{\frac{sp}{2}}\,dx\bigg)^{\frac1s}\\
& \le C\,\e_0^{1-\frac1s}\int_{B_r}(1+|\nabla u|^2)^{\frac p2}\,dx,
\end{align}
where we have used assumption \emph{iii)} and  $C=C(n,p,\beta/\alpha,\nu/L)$ is a positive constant.

\indent Since $E_H$ has flat interface in $B_1$, Theorem~\ref{FlatMorrey}, Remark~\ref{rem:scales} (applied in $B_r^+$ and, similarly, in $B_r^-$) and the minimality of $u_H$ with respect to $u$ yield that
for every $0<\delta<n$,
\begin{equation}\label{eq:flat-decay-uH}
\int_{B_{\tau r}}|\nabla u_H|^p\,dx
\leq C\,\tau^{\,n-\delta}\int_{B_r}(1+|\nabla u_H|^2)^{\frac p2}\,dx\leq C\,\tau^{\,n-\delta}\int_{B_r}(1+|\nabla u|^2)^{\frac p2}\,dx,
\end{equation}
with $C=C(n,p,\beta/\alpha,L/\nu,\delta)$.
Combining \eqref{eq:flat-decay-uH} and \eqref{eqq1} we get
\begin{align}
\int_{B_{\tau r}}|\nabla u|^p\,dx
& \le C\int_{B_{\tau r}}|\nabla u_H|^p\,dx + C\int_{B_\frac{r}{2}}|\nabla u-\nabla u_H|^p\,dx\\
&\le C\Big(\tau^{n-\delta}+\varepsilon_0^{1-\frac{1}{s}}\Big)\int_{B_r}\big(1+|\nabla u|^2\big)^{\frac p2}\,dx
\label{eq:p-decay-u-with-eps}
\end{align}
Choose $\e_0=\e_0(\tau)>0$ so small that $C\,\e_0^{1-\frac1s}\le \tau^{\,n-\delta}$.
Then \eqref{eq:p-decay-u-with-eps} simplifies to
\begin{equation}\label{eq:p-decay-u-clean}
\int_{B_{\tau r}}|\nabla u|^p\,dx
\le C\,\tau^{\,n-\delta}\bigg(\int_{B_r}|\nabla u|^p\,dx+r^n\bigg).
\end{equation}

\medskip
\noindent\textbf{Step 3: the one-phase cases (i) and (ii).}
Assume (i): $|E\cap B_r|<\e_0|B_r|$ (the case (ii) is analogous). Let us denote by $v\in u+W^{1,p}_0(B_\frac{r}{2})$ be the unique minimizer of the functional
\[
\mathcal G(w):=\int_{B_\frac{r}{2}}\,F(\nabla w)\,dx .
\]
We can proceed as in step $(iii)$ subtracting the Euler-Lagrange equation satisfied by $u$ from the one satisfied by $v$ and testing them with $u-v$, getting
\begin{align}
    \int_{B_{\frac{r}{2}}} \sigma_E
\big(DF(\nabla u)-DF(\nabla v)\big)
\cdot(\nabla u-\nabla v)\,dx
& =
\int_{B_{\frac{r}{2}}} (\alpha-\sigma_E)\,
DF(\nabla v)\cdot(\nabla u-\nabla v)\,dx
\end{align}
Using assumptions \eqref{Monotonicity} and \eqref{Growth} and H\"older's inequality, for every $\varepsilon>0$ it holds that
\begin{align}
    \int_{B_{\frac{r}{2}}}|\D v -\D u|^p\,dx
    & \leq C\int_{E\cap B_{\frac{r}{2}}}\big(\mu^2+|\D v|^2\big)^{\frac{p-2}{2}}|\D u-\D v|\,dx\\
& \leq C\bigg(\frac{1}{\varepsilon}\int_{E\cap B_{\frac{r}{2}}}\big(\mu^2+|\D v|^2\big)^{\frac{p}{2}}\,dx+\varepsilon\int_{B_\frac{r}{2}}|\D u-\D v|^p\,dx \bigg).
\end{align}
Choosing $\varepsilon$ sufficiently small, we get 
\begin{align}
    \int_{B_{\frac{r}{2}}}|\D v-\D u|^p\,dx
& \leq C\int_{E\cap B_{\frac{r}{2}}}\big(\mu^2+|\D v|^2\big)^{\frac{p}{2}}\,dx.
\end{align}
Using  the minimality of $v$ compared with $u$ and Theorem
\ref{HigherInt} we deduce,
\begin{align}
\label{eqq11}
\int_{B_{\frac{r}{2}}}|\D v-\D u|^p\,dx& \leq \int_{E\cap B_{\frac{r}{2}}} (1+|\nabla u|^2)^{\frac p2}\,dx\\
& \le \bigg(\frac{E\cap B_{r}}{|B_r|}\bigg)^{1-\frac1s}|B_r|\bigg(\dashint_{B_r}(1+|\nabla u|^2)^{\frac{sp}{2}}\,dx\bigg)^{\frac1s}\\
& \le C\,\e_0^{1-\frac1s}\int_{B_r}(1+|\nabla u|^2)^{\frac p2}\,dx,
\end{align}
where we have used assumption \emph{i)} and  $C=C(n,p,\beta/\alpha,\nu/L)$ is a positive constant. Thereafter we can argue as in the previous step using the classical decay estimate for minimizers of regular integrals in the Calculus of Variations
\[
\int_{B_{\tau \rho}}|\nabla v|^p\,dx
\le C\tau^{n}\int_{B_{ \rho}}|\nabla u|^p\,dx,
\]
for every $\rho\leq \frac{r}{2}$.
\end{proof}

\section{Proof of Theorem \ref{MainThm} }

\begin{proof}
\noindent 
Let $\Omega'\Subset\Omega$. We show that for every $\e>0$ there exists $0<r_*(\e)<d:=\operatorname{dist}(\Omega',\partial\Omega)$ such that for every $x_{0}\in\partial E\cap\Omega'$ there exists a halfspace $H=H(x_{0})\subset\mathbb{R}^{n}$ satisfying
\begin{equation}\label{eq19}
    \frac{|(E\Delta H)\cap B_{r}(x_{0})|}{|B_{r}|}<\e,
\quad\text{for all }0<r<r_*(\e).
\end{equation}

Let us choose $H(x_{0})$ as the halfspace determined by the tangent hyperplane
$T_{x_{0}}\partial E$ and the exterior normal to $E$ at $x_{0}$. Since $\partial E\cap\Omega$ is a $C^{1}$-hypersurface, for each $x_0\in\partial E\cap\Omega'$
(after a rigid motion) $\partial E$ can be represented in a neighborhood of $x_0$ as the graph of a $C^1$ function.
The halfspace $H(x_0)$ approximates $E$ in measure at small scales,
and by compactness of $\partial E\cap\overline{\Omega'}$ and continuity of the unit normal, the approximation can be chosen
uniformly with respect to $x_0\in\partial E\cap\Omega'$.
This yields the claim \eqref{eq19}.\\
\indent Fix $\Omega'\Subset\Omega$ and let $\lambda\in[0,n)$.
Set
\[
\delta:=n-\lambda\in(0,n].
\]
Let $\tau_0\in(0,1)$ be given by Proposition~1.3 and fix $\tau\in(0,\tau_0)$.
Let $\e_0=\e_0(\tau)>0$ be the corresponding threshold in Proposition~1.3. We take $\e=\e_0/2^{n}$ and obtain $r_0:=r_*(\e_0/2^{n})\in(0,d)$ such that for every
$x\in\partial E\cap\Omega'$ and every $0<r<r_0$ there exists a halfspace $H$ with
\begin{equation}\label{eq:halfspace-small-corr}
\frac{|(E\Delta H)\cap B_{r}(x)|}{|B_{r}|}<\frac{\e_0}{2^{n}}.
\end{equation}
Let $x\in\Omega'$ and $0<r<r_0/2$. We distinguish two cases.

\smallskip
\noindent\emph{(a) One-phase case.}
If $B_r(x)\cap\partial E=\emptyset$, then either $B_r(x)\subset E$ or $B_r(x)\subset E^c$.
Hence $|E\cap B_r(x)|=0$ or $|B_r(x)\setminus E|=0$, and condition (i) or (ii) of Proposition~1.3 holds on $B_r(x)$.  Actually, we could even get better regularity, the function $u$ being $p$-harmonic in $B_r(x)$ (see also Remark \ref{Reg}).
\smallskip
\noindent\emph{(b) Two-phase case.}
If $B_r(x)\cap\partial E\neq\emptyset$, choose $x^*\in\partial E\cap B_r(x)$.
Then $B_r(x)\subset B_{2r}(x^*)$ and $B_{2r}(x^*)\Subset\Omega$ since $2r<r_0<d$.
Applying \eqref{eq:halfspace-small-corr} at $x^*$ with radius $2r$, we find a halfspace $H$ such that
\[
|(E\Delta H)\cap B_{2r}(x^*)|<\frac{\e_0}{2^{n}}\,|B_{2r}|=\e_0\,|B_r|.
\]
Therefore,
\[
|(E\Delta H)\cap B_r(x)|\le |(E\Delta H)\cap B_{2r}(x^*)|<\e_0\,|B_r|,
\]
i.e. condition (iii) of Proposition~1.3 holds on $B_r(x)$.\\
\indent In either case, Proposition~1.3 applies on $B_r(x)$ and yields the existence of a positive constant $C_0$ such that
\begin{equation}\label{eq:decay-morrey-corr}
\int_{B_{\tau r}(x)} |\D u|^p\,dy
\le
C_0\,\tau^{n-\delta}\bigg(\int_{B_r(x)}|\D u|^p\,dy
+
r^n\bigg),
\end{equation}
for all $x\in\Omega'$, $\tau<\tau_0$ and all $0<r<r_0/2$. This inequality implies that
\begin{equation}    \sup_{\substack{x\in\Omega'\\\rho\in(0,\tau_0 r)}}\rho^{\delta-n}\int_{B_\rho(x)}|\D u|^p\,dx\leq C\bigg(r^{\delta-n}\int_{B_r(x)}|\D u|^p\,dx+r^n\bigg).
\end{equation}
For every $x\in\Omega'$ and $\rho\in[\tau_0 r,d)$ it holds that
\begin{equation*}
    \rho^{\delta-n}\int_{B_\rho(x)}|\D u|^p\,dx\leq (\tau_0 r)^{\delta-n}\int_{B_d(x)}|\D u|^p\,dx.
\end{equation*}
Thus, taking into account the previous two estimates, it holds that
\begin{equation*}
\sup_{\substack{x\in\Omega'\\\rho\in(0,d)}}\rho^{\delta-n}\int_{B_\rho(x)}|\D u|^p\,dx<+\infty,
\end{equation*}
which leads to the thesis.
\end{proof}

\emph{Acknowledgements} The authors are members of the Gruppo Nazionale per l’Analisi Matematica, la Probabilità e le loro Applicazioni (GNAMPA) of the Istituto Nazionale di Alta Matematica (INdAM).\\

\textbf{Data Availability} Data sharing not applicable to this article as no datasets were generated or analysed during the current study.\\

\textbf{Conflicts of interest} The authors declare that they have no conflicts of interest.

\end{document}